\documentclass[11pt]{amsart}
\pdfoutput=1

\usepackage{hyperref}
\usepackage{amsopn,amssymb,amsmath,amsthm}
\usepackage{eucal}
\usepackage{nicefrac}
\usepackage[all]{xy}
\usepackage{graphicx}
\usepackage{lpic}
\usepackage{float}
\usepackage{enumerate}
\usepackage{subcaption}
\usepackage[margin=.85 in]{geometry}

\allowdisplaybreaks

\newtheorem{theorem}{Theorem}[section]
\newtheorem{lemma}[theorem]{Lemma}
\newtheorem{proposition}[theorem]{Proposition}
\newtheorem{corollary}[theorem]{Corollary}

\newtheorem*{theorem*}{Theorem}
\theoremstyle{remark}
\newtheorem{remark}[theorem]{Remark}

\numberwithin{equation}{section}

\newcommand{\Hb}{\mathbb{H}}

\newcommand{\cI}{\mathcal{I}}        
\newcommand{\vol}{\mathrm{vol}}

\title{Packing curves on surfaces with few intersections}
\author{Tarik Aougab, Ian Biringer, Jonah Gaster}

\address{Department of Mathematics, Brown University \\
Providence RI 02912, USA}
\email{tarik\_aougab@brown.edu}

\address{Department of Mathematics, Boston College \\
Chestnut Hill MA 02467, USA}
\email{biringer@bc.edu}

\address{Department of Mathematics, Boston College \\ 
Chestnut Hill MA 02467, USA}
\email{gaster@bc.edu}

\begin{document}
\date{October 20, 2016}

\begin{abstract}
Przytycki has shown that the size  $\mathcal{N}_{k}(S)$ 
of a maximal collection of simple closed curves that pairwise intersect at most $k$ times on a topological surface $S$ grows at most as a polynomial in $|\chi(S)|$ of degree $k^{2}+k+1$. 
In this paper, we narrow Przytycki's bounds by showing that 
\[ 
\mathcal{N}_{k}(S) =O \left( \frac{ |\chi|^{3k}}{ ( \log |\chi| )^2 } \right) ,
\]
In particular, the size of a maximal 1-system grows sub-cubically in $|\chi(S)|$. 
The proof uses a circle packing  argument of Aougab--Souto and a bound for the number of curves of length at most $L$ on a hyperbolic surface.
When the genus $g$ is fixed and the number of punctures $n$ grows, we can improve  our estimates using a different argument to give $$\mathcal{N}_{k}(S) \leq O(n^{2k+2}).$$ Using similar techniques,  we also obtain the sharp estimate
$\mathcal{N}_{2}(S)=\Theta(n^3)$ when $k=2$ and $g$ is fixed.
 \end{abstract}

\maketitle

\section{Introduction}
Let $S=S_{g,n}$ be an oriented surface of genus $g$ with $n$ punctures, and set $\chi=\chi(S)$.
Given $k \in \mathbb{N}$, a \textit{k-system} of curves (resp.~arcs) is a collection of pairwise non-homotopic simple closed curves (resp.~properly embedded arcs) on $S$, no two of which intersect more than $k$ times.  
Let 
\begin{align*}
\mathcal{A}_k(S) &= \max \left\{ \ |\Gamma| :  \ \Gamma \text{ is a $k$-system of arcs on $S$} \ \right\} , \text{ and} \\
\mathcal{N}_k(S) &= \max \left\{ \ |\Gamma| : \ \Gamma \text{ is a $k$-system of curves on $S$} \ \right\} .
\end{align*}
In 1995, Juvan--Malni\v c--Mohar \cite{Juvansystems} showed that $\mathcal{N}_k(S)$ is always finite. The asymptotic study of $\mathcal{N}_k(S)$ as $|\chi|\to \infty$ was later popularized by  Benson Farb and Chris Leininger, who vocally noticed that good bounds were unavailable even when $k=1$.

 In response, Malestein--Rivin--Theran \cite{m-r-t}  showed that when $S$  is closed and any $k$ is fixed, $\mathcal{N}_k(S)$  grows  at least quadradically and at most exponentially in $|\chi|$.  Observing that  asymptotics are easier to find for arcs,  and that the arc case  informs the curve case, Przytycki  \cite{przytycki} then showed that for fixed $k$ \begin{equation}
 	\mathcal{A}_k(S) = \Theta \left( |\chi|^{k+1} \right), \ \ \text{and} \ \ \mathcal{N}_k(S) = O\left( |\chi|^{k^2+k+1} \right). \label {przbounds}
 \end{equation}

The first author has shown \cite{aougab2015local} that  \emph{when $S$ is fixed,} $\log(\mathcal{N}_{k}(S))$ grows at most linearly in $k$, which suggests that it might be possible to improve Przytycki's upper bound for curves. 
We show:

\begin{theorem}
	
\label{k-system bound} Suppose $S$  is a surface with Euler  characteristic $\chi$. Then as $|\chi|\to \infty$ we have
\[
\mathcal{N}_k(S) 
\le O \left( \frac{ |\chi|^{3k}}{ ( \log |\chi| )^2 } \right) ~.
\]
\end{theorem}

 Note that when $k=1$ and $S$ is closed, this is a slight improvement of Przytycki's bound of $O(|\chi|^3)$. (For larger $k$, though, the improvement is significant.) 
The best known constructions of $1$-systems show that $\mathcal{N}_{1}(S)$ grows at least quadratically in $|\chi|$. Thus it is of course natural to ask whether there exists an $\epsilon>0$ so that $\mathcal{N}_1(S) = O(|\chi|^{3-\epsilon})$.
Indeed, one might expect in general that $\mathcal{N}_k(S) = \Theta(\mathcal{A}_k(S)) = \Theta(|\chi|^{k+1})$, but currently this is out of reach.

The proof of  Theorem \ref{k-system bound}, which appears in \S\ref{proof of main thm}, is short enough that there is no need for a detailed summary here.
Nonetheless, we point out three main components.
The first is (a trivial adaptation of) an argument of Aougab--Souto \cite[Theorem 1.2]{aougab-souto},  in which a circle packing argument is used to  find a hyperbolic structure on which a given $k$-system can be realized length-efficiently.  
The other components are Theorems \ref{improve k-thm} and \ref{linear growth} below, which we consider of independent interest.

 We should highlight that this argument uses Przytycki's bound  $\mathcal{A}_k(S)= \Theta \left( |\chi|^{k+1} \right)$ 
essentially, in the proof of Theorem \ref{improve k-thm}. It turns out that we can exploit his bounds even more effectively when the number of punctures is large in comparison to the genus. Here, the proof is inductive, where we relate $\mathcal N_k(S_{g,n})$ to $\mathcal N_k(S_{g,n-1})$ by  projecting a $k$-system on $S_{g,n}$ to one on $S_{g,n-1}$  by filling in the puncture. We note that Malestein-Rivin-Theran \cite[Thm.~1.2]{m-r-t} used a similar inductive argument in the $k=1$ case to show that $\mathcal{N}_1(S_{g,n}) = \mathcal{N}_1(S_{g,0}) + C g\cdot n.$

\begin{theorem}
\label{2-system bound}
There is a constant $C=C(k)$ such that $$\mathcal{N}_k(S_{g,n}) \leq  \mathcal N_k(S_{g,0}) + C (g+n)^{2k+2}.$$ And in fact, for $k=2$ we have
$$\displaystyle \mathcal{N}_2(S_{g,n}) \leq \mathcal{N}_2(S_{g,0}) + C(g+n)^{3}.$$
\end{theorem}

When $k$ is even, Przytycki's construction \cite[Example 4.1]{przytycki} of large $k$-systems of arcs can be tweaked in a straightforward manner to produce the lower bound $\approx (g+n)^{k+1} / (k+1)^{k+1}$ for $\mathcal{N}_k(S)$. So,

\begin {corollary}
 When $g$ is fixed and $n\to \infty$, we have $\mathcal{N}_2(S_{g,n})=\Theta(n^3).$
\end {corollary}

 \subsection{Degree bounds}

Let $\mathcal{I}(\Gamma)$ denote the \textit{intersection graph} of a curve system $\Gamma$, whose vertices are in 1-1 correspondence with the curves in $\Gamma$,  and where two vertices are connected by an edge exactly when the corresponding curves intersect essentially on the surface $S$. 

In \cite{przytycki}, Przytycki's estimate for $ \mathcal{N}_k(S)$  is a corollary of his estimate for $ \mathcal A_k(S)$. The idea is as follows, say when $k=1$. If $\Gamma$  is a $1$-system, cut $S$ open along a curve $\gamma \in \Gamma$.  Any curve in $\Gamma$ intersecting $\gamma$ becomes an arc on the new surface $S'$.  The number of homotopy classes of such arcs is bounded above by $\mathcal{A}_{1}(S')$, and at most two curves in $\Gamma$ correspond to the same homotopy class of arc on $S'$. It follows that the degree of $\gamma$ in $\mathcal{I}(\Gamma)$ is at most $2 \cdot \mathcal{A}_{1}(S')$. To finish, note that the total number of vertices in $\mathcal{I}(\Gamma)$ is at most the sums of the degrees of the vertices in a maximal independent subset of $\mathcal I(\Gamma)$, and any  independent set (i.e.\  a set of disjoint simple closed curves on $S$) has size at most linear in $\chi$.

 To prove Theorem \ref{k-system bound} in the closed case,  we will need the following sharper upper bound for the degree of a vertex in the intersection graph of a $k$-system. 

\begin{theorem}
\label{improve k-thm}
Suppose that $\Gamma$ is a $k$-system on  a surface $S$, and $\gamma \in \Gamma$. Then  the degree of $\gamma$ in the graph $\mathcal{I}(\Gamma)$  is at most $C \cdot |\chi|^{3k-1}$,  for some universal $C=C(k)$.
\end{theorem}

When $k=1$,   this bound agrees with the bound $2 \cdot \mathcal{A}_{1}(S')=\Theta(|\chi|^{2})$ one gets with Przytycki's argument above and \eqref{przbounds}.  In general, though, his argument gives $C \cdot |\chi|^{k(k+1)}$, so  Theorem \ref{improve k-thm} is  quite a bit stronger.
The improvement arises by adopting a slightly different perspective.
Instead of cutting open $S$ along a curve to produce an arc system on a surface of smaller complexity, one can introduce punctures to $S$ and `slide' curves to arcs to arrive at an arc system on a surface $S'$ of slightly larger complexity, and then apply Przytycki's bounds for $\mathcal{A}_k(S')$.

\subsection{Counting curves when the  surface varies}\label {hypgeomintro}

Fix a hyperbolic structure $X$ on $S$, and let $\mathcal G_X(L)$ be the number of primitive closed geodesics on $X$  with  length at most $L$. 
We prove:

\begin{theorem}
\label{linear growth}
For any hyperbolic structure $X$ on $S$ and any $L>0$, we have
\[
\mathcal{G}_X(L) \le \left( \frac{3}{2} + \frac{\pi^2}{\mu}\cdot e^{2L} \right) \cdot |\chi|~,
\]
where $\mu \approx .2629 $ is the Margulis constant for $\mathbb{H}^2$.  
\end{theorem}

Note that when $X$ is fixed and $L\to \infty$, classical work of Delsarte, Huber and Selberg (see \cite{Busergeometry}) implies that $\mathcal G_X(L) = \Theta(e^L/L)$, but such asymptotics do not give uniform bounds of the sort above.  

 To prove Theorem \ref{linear growth}, we first use  a geometric  interpretation of J\o rgensen's inequality, due to Gilman~\cite{gilman}, to say that geodesics of length  at most $L$ cannot  get too close in the unit tangent bundle $T^1X$.  Small neighborhoods of these geodesics must then be disjoint, and Theorem \ref{linear growth}  follows since the sum of their volumes is less than the total volume of $T^1X$, which is $4\pi^2 |\chi|$. See \S \ref{sec:linear}.

While Theorem \ref{linear growth} is certainly not sharp, it is not so far from optimal.  Namely, one can show that  there is a constant $C>0$ such that $S$ supports a hyperbolic structure $X$ with
\begin{equation}C \cdot e^{\frac{L}{4}} \cdot |\chi| \leq 
\mathcal{G}_X(L).\label{twist}\end{equation}
In fact, one can realize this lower bound using \emph{simple} closed curves.  The point is that in the chosen hyperbolic structure $X$,  many curves are pinched: a large  number of simple closed curves with controlled lengths can be obtained by twisting around the pinched curves. 
See Proposition \ref{short curves} in \S\ref{sec:linear}.

In our application, what is essential is that the upper bound above is at most $C^L|\chi|$ for some constant $C$, but it would be interesting to see how much  the gap between Theorem \ref{linear growth} and \eqref{twist}  could be narrowed. It seems likely that an upper bound on the order of $e^{(3/2) L}|\chi|$ could be given for  the number of \emph {simple}  closed curves of length at most $L$,  using the Collar Lemma \cite{farb-margalit}  in place of J\o rgensen's inequality. And when the genus of $S$  is linear in $|\chi|$, e.g.\ when $S$ is closed, the lower bound \eqref{twist} can be improved to $C e^{L/2} |\chi|$, again with simple curves (see Remark \ref{one-holed tori}).


\section{Hyperbolic geometry lemmas and the proof of Theorem \ref{linear growth}}
\label{sec:linear}
We fix a hyperbolic surface $X$ in what follows, and indicate the unit tangent bundle of $X$ by $T^1X$.
Recall Gilman's geometric reinterpretation of J\o rgensen's inequality:

\begin{theorem}\cite[p.~5]{gilman}
If $\alpha$ and $\beta$ are a pair of geodesics on $X$ of length at most $L$, then they are either disjoint and
\begin{align}
\label{eq:distance}
d_X(\alpha,\beta) \ge \sinh^{-1} \left( \frac{2}{\sinh^2 (L/2)} \right)~,
\end{align}
or their minimal angle of intersection $\theta(\alpha,\beta)$ satisfies
\begin{align}
\theta(\alpha,\beta) \ge \sin^{-1} \left( \frac{1}{\sinh^2 (L/2)} \right)~.
\end{align}
\end{theorem}

Motivated by Gilman's theorem, let 
\[
\epsilon(L) = \frac{1}{2} \min 
\left\{  \sinh^{-1} \left( \frac{2}{\sinh^2 (L/2)} \right) , 
\sin^{-1} \left( \frac{1}{\sinh^2 (L/2)} \right) \right\}~.
\]
For a geodesic $\gamma$ on $X$, let $N_\epsilon(\gamma) \subset T^1X$ indicate the $\epsilon$-neighborhood of $T^1 \gamma \subset T^1 X$.
The following is an immediate consequence of Gilman's theorem: 

\begin{lemma}
\label{tubular neighborhoods}
If $\alpha$ and $\beta$ are closed geodesics on $X$ of length at most $L$, then the tubular neighborhoods $N_{\epsilon(L)}(\alpha)$ and $N_{\epsilon(L)}(\beta)$ are disjoint.
\end{lemma}

A straightforward calculation in the hyperbolic plane shows:
\begin{lemma}
\label{volume}
When $\epsilon \le \pi/2$, the volume $\vol \ N_\epsilon(\gamma)$ is given by
$8 \cdot \ell_X(\gamma) \cdot (\cosh \epsilon -1)$.
\end{lemma}

We may now prove Theorem \ref{linear growth}:

\begin{proof}[Proof of Theorem \ref{linear growth}]
Suppose that $\gamma_1,\ldots,\gamma_m$ are simple closed curves on a hyperbolic surface $X$ of length at most $L$.
By Lemma \ref{tubular neighborhoods}, the choice $\epsilon =\epsilon(L)$ makes the sets 
$N_{\epsilon}(\gamma_1) ,\ldots, N_{\epsilon}(\gamma_m)$ disjoint.
Now Lemma \ref{volume} implies
\[
\vol(T^1X) \ge \sum_{i=1}^{m} \vol(N_{\epsilon}(\gamma_i)) 
= 8(\cosh \epsilon - 1 ) \sum_{i=1}^{m} \ell_X(\gamma_i)~.
\]
Since $\vol(T^1X) = 4\pi^2|\chi(S)|$ and $2(\cosh x -1) \ge x^2$, we find that 
\[
\sum_{i=1}^{m} \ell_X(\gamma_i) \le \frac{\pi^2 |\chi(S)| }{2(\cosh \epsilon -1)} \le \frac{\pi^2}{ \epsilon^2}\ |\chi(S)|~.
\]

Recall the following well known consequence of the Collar Lemma \cite[Lemma 13.6]{farb-margalit}: any pair of geodesics of length at most $\mu$ are disjoint, where $\mu$ is the Margulis constant $\mu$ for $\Hb^2$. 
It follows that any collection of curves of length less than $\mu$ cannot have more than $3g-3+n$ elements.
Partitioning $\gamma_1,\ldots,\gamma_m$ into those curves shorter and longer than $\mu$, respectively, we find that 
\[
\sum_i \ell_X(\gamma_i) \ge \left( m - \frac{3}{2} |\chi(S)| \right) \cdot \mu~.
\]
Together, these bounds imply that 
\[
m \le \left( \frac{3}{2} + \frac{\pi^2}{\epsilon^2 \mu} \right) |\chi(S)|~.
\]
Finally, the chosen value for $\epsilon$ satisfies $\epsilon \ge e^{-L}$, and the claimed upper bound follows.
\end{proof}

The upper bound of Theorem \ref{linear growth} is reasonable considering the following lower bound:

\begin{proposition}
\label{short curves}
There is a constant $C$ so that $C \cdot e^{\frac{L}{4}} \cdot |\chi| \leq \mathcal{G}_X(L)$.
\end{proposition}

\begin{proof}
Consider first the unique pants curve $\alpha$ in a pants decomposition of a four-holed sphere $F\subset S$, and moreover fix a choice of `zero twisting' so that there is a geodesic $\beta$ which intersects $\alpha$ twice orthogonally.
When a hyperbolic structure is chosen with the length of $\alpha$ given by $r$, by the Collar Lemma $\beta$ has length at least $4\log(1/r)$.
Evidently, if $r$ is chosen much smaller than $e^{-L/4}$, there are \emph{no} nonperipheral simple geodesics supported on $F$ with length at most $L$ other than $\alpha$.
We will choose instead $r \approx e^{-L/4}$ that will allow the construction of many other `short' simple geodesics.

Consider the curve $\beta_n$ formed by applying $n$ half-Dehn twists around $\alpha$ to $\beta$. 
The hyperbolic length of $\beta_n$ in a hyperbolic structure where $\alpha$ has length $r$ is at most
\[
n \cdot r + 4 \log \left( \frac{1}{r} \right) + C_0~, 
\]
where $C_0$ is some uniform constant determined by the hyperbolic geometry of three-holed spheres with geodesic boundary.

Now let $r$ be chosen as $e^{-\frac{L}{4}+C_0}$. The length of $\beta_n$ is at most 
\[
n\cdot e^{-\frac{L}{4}+C_0} + L - 3C_0~,
\]
which is in turn less than $L$ as long as $n$ is at most $C_1e^{L/4}$, where $C_1=3C_0e^{-C_0}$.

Since there are some constant proportion of $|\chi(S)|$ many subsurfaces which are four-holed spheres, we have constructed $C\cdot e^ {L/4} \cdot |\chi(S)|$ curves of length at most $L$, as desired.
\end{proof}

\begin{remark}
\label{one-holed tori}
The above construction in one-holed tori subsurfaces would produce $Ce^{L/2}$ curves of length at most $L$, with the difference attributable to the fact that one can build curves that only cross the long annulus once.
On the other hand, planar surfaces have no one-holed tori subsurfaces, so this would only produce a better lower bound for surfaces with high genus.
\end{remark}

\begin{remark}
\label{non-simple remark}
The upper bound above applies verbatim to non-simple curves. 
A tighter bound could likely be given using the Collar Lemma to deal with the simple case, since the latter would produce a better lower bound in \eqref{eq:distance}.
This is complicated by the fact that the tubular neighborhoods $N_{\epsilon(L)}(\gamma_i)$ will not suffice for an improved bound, and a more careful construction of neighborhoods is required that use more of the embedded annuli whose presence is guaranteed by the Collar Lemma.
\end{remark}


\section{Proof of Theorem \ref{k-system bound} in the closed case}
\label{proof of main thm}

Let $\Gamma = \{\gamma_1, \ldots, \gamma_N\}$ be a $k$-system of curves on a closed surface $S$ with $|\chi(S)|=t$. 
We begin by using the following result of the first author and Souto.

\begin{proposition}\cite[Thm.~1.2]{aougab-souto}
\label{circle packing lemma}
There exists a hyperbolic structure $X$ on $S$  such that the geodesic realization of $\Gamma$ on $X$ has total length
\[
\ell_X(\Gamma) \le 4 \sqrt{ 2 t \cdot \iota(\Gamma,\Gamma) }~,
\] 
where $\iota(\Gamma,\Gamma)$ is the total geometric self-intersection number of $\Gamma$.
\end{proposition}

 The idea behind Proposition \ref{circle packing lemma} is as follows. By Koebe's Discrete Uniformization Theorem, there exists a hyperbolic structure $X$ on $S$ so that the union of all the curves in $\Gamma$  can be realized inside the dual graph of a circle packing on $X$.  The sum of the areas of these circles is at most $2\pi t$, by Gauss Bonnet.  Using the Cauchy-Schwarz Inequality, this translates into an upper bound on the sum of the radii of the circles, which bounds the length of $\Gamma$. 
This argument is carried out for self-intersecting curves in \cite{aougab-souto}, but it applies verbatim to our setting of curve systems.

Now since $\Gamma$  is a $k$-system, we have 
\begin{equation*}
|\iota(\Gamma,\Gamma)| \le k \sum_{i=1}^N \mathrm{deg}_{\cI(\Gamma)}(\gamma_i) \leq C \cdot t^{3k-1}\cdot N,
\end{equation*}
where $\mathrm{deg}_{\cI(\Gamma)}$  is  the degree of $\gamma$ in the intersection graph $\cI(\Gamma)$, i.e.\ the number of curves in $\Gamma$ that $\gamma$ intersects, $C=C(k)$ is a constant, and the last inequality is Theorem \ref{improve k-thm}.

Using Proposition \ref{circle packing lemma}, the \emph{average length} of a curve from $\Gamma$ is then \[
\frac{\ell_X(\Gamma)}{N} \le C \sqrt{ \frac{ t^{3k} }{N}} =: L,
\] 
 for some new $C=C(k)$.
 By Chebyshev's inequality, at least half of $\ell_X(\gamma_1),\ldots,\ell_X(\gamma_N)$ are less than or equal to twice the average length, so by Theorem \ref{linear growth} we have
\begin{equation}
	N \le 2 \cdot \# \left\{ \text{closed geodesics }\gamma \text{ on }X : \ell_X(\gamma)\le 2L \right\} \leq e^{C\sqrt{ t^{3k} /N }} t~,\label{firsteqpf}
\end{equation}
for some new constant $C=C(k)$. 
If we suppose that $N\neq O \left( \frac{t^{3k}} { (\log t)^2 } \right)$,  
then we can find  arbitrarily large $t$  for which there is a $k$-system with $N \geq \frac{C^2t^{3k}} { (\log t)^2 }$ curves. 
Then  \eqref{firsteqpf} says 
\[
 \frac{C^2t^{3k}} { (\log t)^2 } \leq N \leq e^{C\sqrt{t^{3k}/{N}}} \cdot t \leq e^{C\sqrt{t^{3k}/{ \frac{C^2t^{3k}} { (\log t)^2 }}}} \cdot t = t^2~,
\]
a contradiction when $k>0$ and $t$ is large. 
This concludes the proof of Theorem \ref{k-system bound}.

\begin{remark}
\label{turan remark}
We recall an observation of Malestein-Rivin-Theran:
It is evident that the size of an independent set of $\mathcal{I}(\Gamma)$ is bounded above linearly in $t$, and an application of Tur{\'a}n's theorem to the complementary graph of $\cI(\Gamma)$ implies 
\[
|\Gamma| \leq  3t/2 \cdot (D+1)~,
\]
where $D$ is the average degree of $\cI(\Gamma)$ \cite[Theorem 1.5]{m-r-t}.
Coarsely, $|\Gamma| = O(t\cdot D)$. Arguments similar to those above can also be used to show the stronger statement $$\displaystyle |\Gamma| = O \left( \frac{ t \cdot D } {(\log g)^2} \right)$$ for maximal $k$-systems $\Gamma$.
\end{remark}


\section{Proof of Theorem \ref{improve k-thm}}
\label{proof of imrpove k-thm}

In this section we prove Theorem \ref{improve k-thm}. Let $\Gamma$ be an arbitrary $k$-system on $S$ with genus $g$ and let $\gamma \in \Gamma$. 
We will show that the set of curves $\mathcal{I}(\gamma)$ consisting of elements of $\Gamma$ intersecting $\gamma$ non-trivially, has size $O(g^{3k-1})$. 
Begin by choosing a realization of $\mathcal{I}(\gamma) \cup \gamma$ with no triple points, and pick an orientation and a basepoint $x$ on $\gamma$.   
Then the intersections of $\gamma$ with the curves in $\mathcal{I}(\gamma)$ are ordered according to when they appear when one traverses $\gamma$ in the given direction starting at $x$.

Let $S_x$ be the surface obtained by puncturing $S$ at $x$. 
If $\alpha \in \mathcal{I}(\gamma)$, we produce an arc $\tilde \alpha$ on $S_x$ as follows. Let $y$ be the first intersection point of $\alpha$ and $\gamma$, with respect to the order above. Isotope $\alpha$ by pushing $y$ along $\gamma$ to $x$,  in the direction opposite to the orientation.  This gives an arc $\tilde \alpha $ on $S_x$,
$$\tilde \alpha =\gamma_{\alpha}^{-1} \ast {\alpha} \ast \gamma_{\alpha},$$ where $\gamma_{\alpha}$ is the directed sub-arc of $\gamma$ from $x$ to $y$.    
Since $y$  was the \emph {first}  intersection point along $\gamma$ from $x$,  the arc $\tilde \alpha$ can be perturbed to be simple. (Unperturbed, it tracks $\gamma_\alpha$ twice.) 

 We claim that when $\alpha,\rho \in \mathcal{I}(\gamma)$, we have $\iota( \tilde \alpha, \tilde \rho) \leq 3k-2$. Suppose that  with respect to the order above, $\gamma$  intersects $\alpha $ before it intersects $\rho$.  Then no new intersections with $\rho$ are created when $\alpha$  is replaced by $\tilde \alpha$. Moreover, when pushing $\rho$ along $\gamma$ to create $\tilde \rho$, one may encounter at most $k-1$ strands of $\tilde{\alpha}$. This follows from the fact that $\iota(\alpha, \gamma) \leq k$ and that the first intersection point between $\alpha$ and $\gamma$ has been pushed all the way to $x$ in the construction of $\tilde{\alpha}$, leaving at most $k-1$ strands. Each such strand will contribute to two intersection points between $\tilde{\alpha}$ and $\tilde{\rho}$, and as $\rho, \alpha$ both belong to a $k$-system, we had $\iota(\alpha,\rho) \leq k$ to begin with. Thus $\iota(\tilde \alpha, \tilde \gamma)  \leq 2(k-1)+ k= 3k-2$, as desired.  

As no two of these arcs $\tilde{\alpha}, \tilde{\rho}$ can be homotopic, we have a $(3k-2)$-system of arcs on $S_x$ with size  equal to $|\cI(\gamma)|$. By Przytycki's upper bound \eqref{przbounds} for arc systems, 
$ |\mathcal{I}(\gamma)| = O(g^{3k-1}) $. This completes the proof of Theorem \ref{improve k-thm}.


\section{The proof of Theorem \ref{2-system bound}}
\label{sec: punctures}

Let $S_{g,n}$ be  a closed, orientable surface with genus $g$ and $n$ punctures.  We first show
\[
{\mathcal{N}}_{k}(S_{g,n}) \leq {\mathcal{N}}_{k}(S_{g, n-1}) + \mathcal A_{k-1}(S_{g,n}) +  \mathcal A_{2k}(S_{g,n})\leq  {\mathcal{N}}_{k}(S_{g,n-1})  + C (g+n)^{2k+1}~,
\]
where $C=C(k)$ is a constant coming from Przytycki's \cite{przytycki} bounds $ \mathcal A_{2k}(S) = \Theta(|\chi|^{2k+1})$. 
Applying this step iteratively, we may conclude 
\[
\mathcal{N}_{k}(S_{g,n}) \leq \mathcal{N}_k(S_{g}) +  \sum_{i=1}^n C (g+i)^{2k+1} = \mathcal{N}_k(S_{g,0}) +C(g+n)^{2k+2}~,
\]
for some new $C=C(k)$.

So, let $\Gamma$ be a $k$-system on $S_{g,n}$.  Fix a minimal position realization of $\Gamma$, and choose arbitrarily a puncture $p$ of $S_{g,n}$. Project each curve in $\Gamma$ to a curve on $S_{g,n-1}$ by  filling in the puncture $p$.  

 We first bound the number of curves that have inessential projection to $S_{n-1}$.  Each such curve $c$ bounds a  twice-punctured disk,  where one of the punctures is $p$.  In other words, $c$ is the boundary of a regular neighborhood of an arc $\alpha_c$ from $p$ to some other puncture. If $c,d$  both have inessential projections, then  $\iota (c,d) = 4\iota (\alpha_c,\alpha_d)+\epsilon$, where $\epsilon=4$ or $\epsilon=2$, depending on whether $c,d$ share the same second puncture  or not.   In particular, $\iota (\alpha_c,\alpha_d) \leq k-1$,  so the  set of all $\alpha_c$  is a $(k-1)$-system of arcs.   Therefore, the number of curves $c$ with inessential projection is bounded  above by $\mathcal A_{k-1}(S_{g,n}).$

Remove all curves from $\Gamma$  that have inessential projection to $S_{g,n-1}$. It suffices to show that now
$$|\Gamma| \leq {\mathcal{N}}_{k}(S_{g,n-1}) + \mathcal A_{2k}(S_{g,n}).$$

 We would like to relate $|\Gamma|$  to ${\mathcal{N}}_{k}(S_{g,n-1})$  by saying that the projection of $\Gamma$ to $S_{g,n-1}$  is a $k$-system. The problem is that curves can become homotopic after projection,  so the size of the new $k$-system can be smaller.  Let $ \mathcal G\subset \Gamma$  be a {maximal} subset of `good' curves whose projections are not homotopic, and let $\mathcal B = \Gamma \setminus \mathcal G$ be the complementary set of `bad' curves. As $| \mathcal G| \leq {\mathcal{N}}_{k}(S_{g,n-1})$,  it suffices to prove
$$| \mathcal B| \leq  \mathcal A_{2k}(S_{g,n}).$$

Let $c\in \mathcal B$  be a bad curve.  There is then a unique good curve $g_c$  such that the projections of $c$  and $g_c$ to $S_{g,n-1}$ are homotopic. It follows that on $S_{g,n}$, the curves $c$  and $g_c$  either:
\begin {enumerate}
\item are disjoint and bound an annulus $A_c$  punctured by $p$,
\item intersect,  and there are arcs of $c$  and $g_c$  that bound a bigon $A_c$ punctured by $p$.
\end {enumerate}
 We refer to bad curves where the former holds as \emph{type (1)} and the other as \emph{type (2)}.
In both cases, define an arc $\alpha_c$  on $S_{g,n}$  by connecting $p$ to $c$  with an arc $\beta$ in $A_c$,  and then  setting
$$\alpha_c = \beta^{-1} \star c\star \beta.$$
Both endpoints of $\alpha_c$ are at $p$, and $\alpha_c$  is well defined up to homotopy.

The desired conclusion follows immediately from the following lemma.

\begin{lemma}
\label{bad arc system}
The set $\{\alpha_c : c\in \mathcal{B} \}$ is a $2k$-system of arcs.
\end{lemma}

\begin{proof}
Fix a pair $c,d \in \mathcal B$, the corresponding curves $g_c,g_d\in\mathcal G$, the arcs $\alpha_c,\alpha_d$, and the regions $A_c$, $A_d$.

We start by assuming that $c$ and $d$ are of type (1).
The component of the intersection $A_c \cap A_d$ containing $p$ is evidently a disk, with at least one side bounded by an arc of either $d$ or $g_d$. 
If this arc does not have its endpoints on the same component of $\partial A_c$, we are in the setting of Figure \ref{annulus1}; if it does we are in the setting of Figures \ref{annulus2} or \ref{annulus3}.
Whichever of $d$ or $g_d$ is on the boundary of the bigon pictured in Figure \ref{annulus2}, it may not intersect $A_c$ in more than $k-2$ other arcs, and it follows that $\iota(\alpha_c,\alpha_d)  \le k-2$, as the two pictured intersections can be homotoped away.
In either of Figures \ref{annulus1} or \ref{annulus3}, there is a realization of $\alpha_d$ pictured that has at most one fewer intersection point with $\alpha_c$ than $c$ has with $d$, so that $\iota(\alpha_c,\alpha_d) \le \iota(c,d) - 1 \le k-1$.

The argument is the same when $c$ is of type (1) but $d$ is of type (2), with the minor complication that there is one new possibility in Figure \ref{annulus2}. 
Note that it is still obvious that the pictured representative for $\alpha_d$ satisfies $\iota(\alpha_c,\alpha_d) \le k$.

Finally, suppose both $c$ and $d$ are of type (2).
The component of $A_c \cap A_d$ containing $p$ is either as in Figures \ref{bigon1} or \ref{bigon2}, or (after homotoping $A_d$ if necessary, leaving it in minimal position with $A_c$) it is as in Figure \ref{bigon3}.
The worst of these cases is Figure \ref{bigon2}, where we see that $\iota(\alpha_c , \alpha_d) \le 2k$. \qedhere

\begin{figure}
	\centering
	\begin{minipage}{.45\textwidth}
	\centering
	\begin{lpic}{annulus0(4cm)}
		\lbl[]{-15,60;$\alpha_c$}
	\end{lpic}
	\subcaption{$c$ is of type (1)}
	\label{annulus0}
	\end{minipage}\hfill
	\begin{minipage}{.45\textwidth}
	\centering
	\begin{lpic}{bigon0(5cm)}
		\lbl[]{65,60;$\alpha_c$}
	\end{lpic}
	\subcaption{$c$ is of type (2)}
	\label{bigon0}
	\end{minipage} \hfill
	\caption{The region $A_c$ cobounded by $c$ and $g_c$, and the bold arc $\alpha_c$.}
\end{figure}

\begin{figure}
	\centering
	\begin{minipage}{.32\textwidth}
	\centering
	\includegraphics[width=4cm]{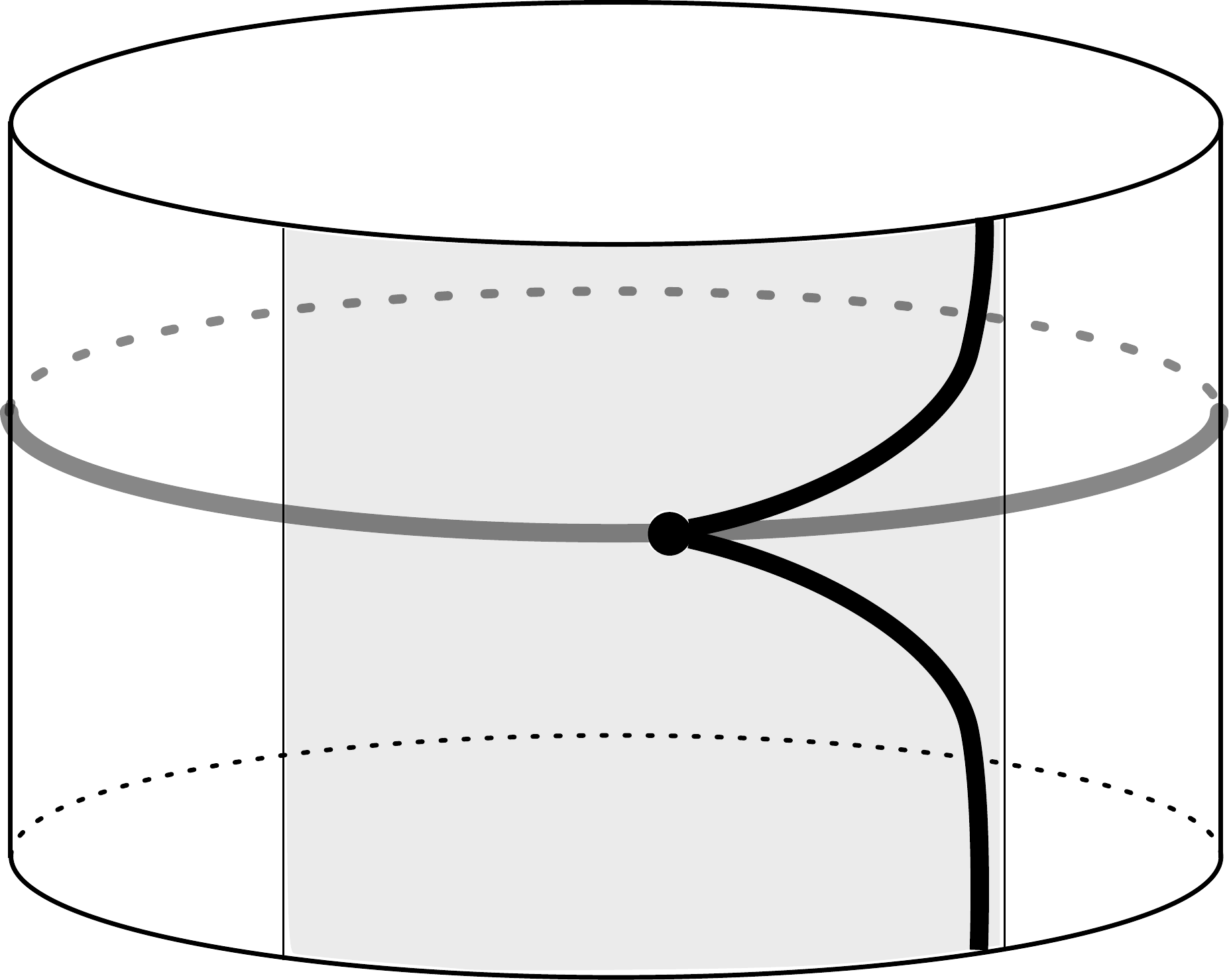}
	\subcaption{}
	\label{annulus1}
	\end{minipage} \hfill
	\begin{minipage}{.32\textwidth}
	\centering
	\includegraphics[width=4cm]{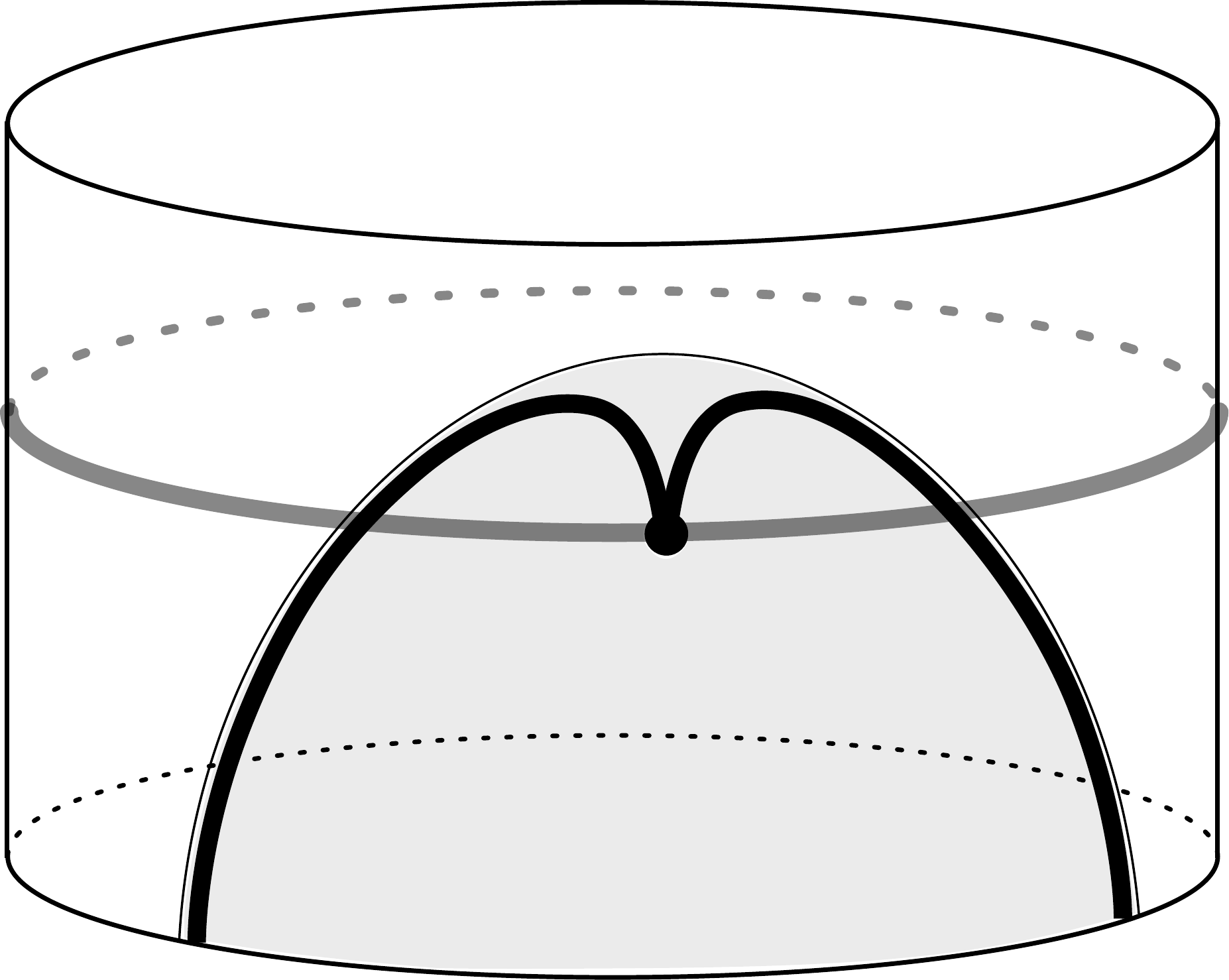}
	\subcaption{}
	\label{annulus2}
	\end{minipage}\hfill
	\begin{minipage}{.32\textwidth}
	\centering
	\includegraphics[width=4cm]{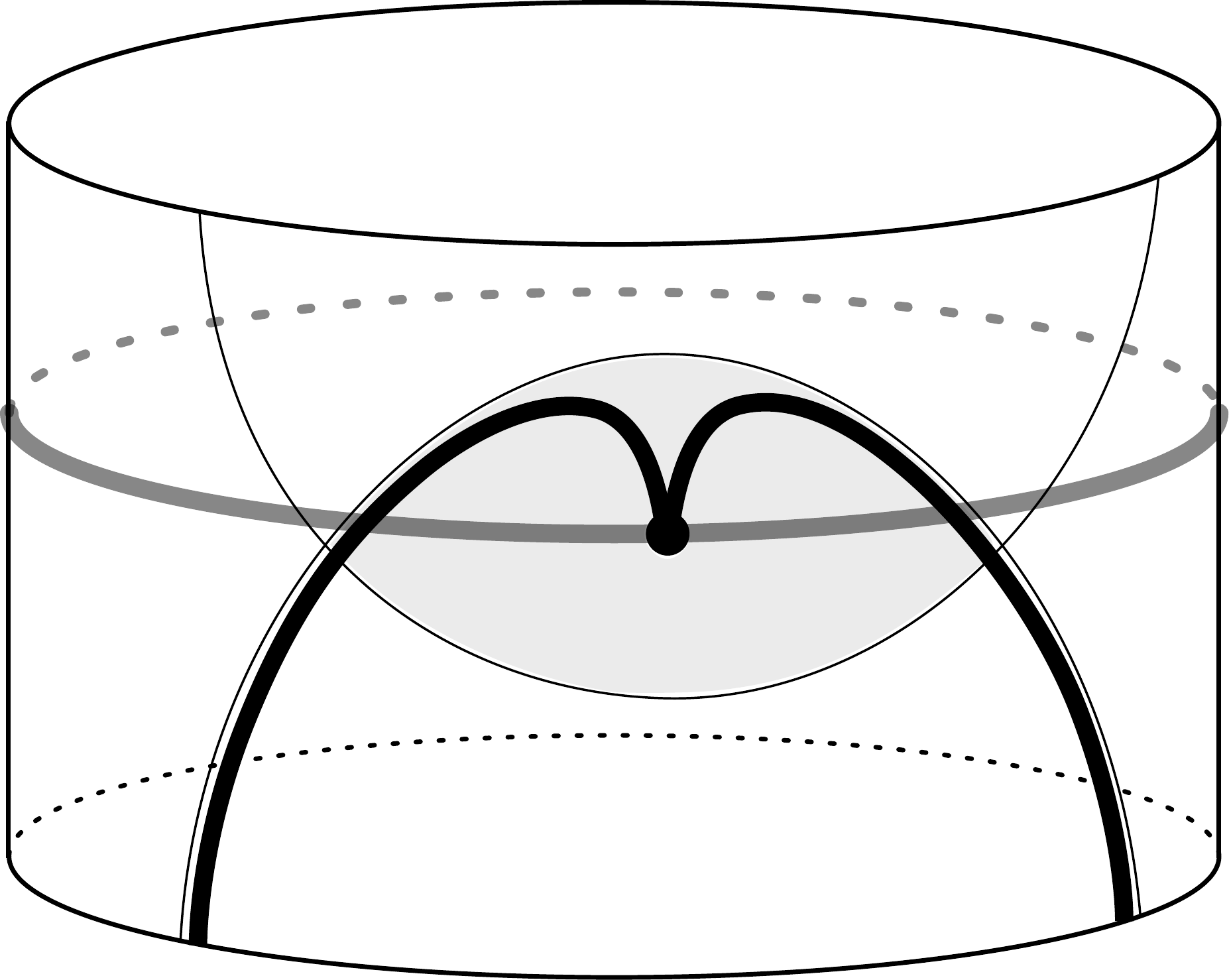}
	\subcaption{}
	\label{annulus3}
	\end{minipage}\hfill
	\caption{Possibilities for $A_c \cap A_d$ when $c$ is of type (1), with $\alpha_c$ shaded and $\alpha_d$ in bold.}
\end{figure}

\begin{figure}
	\centering
	\begin{minipage}{.32\textwidth}
	\centering
	\includegraphics[width=5cm]{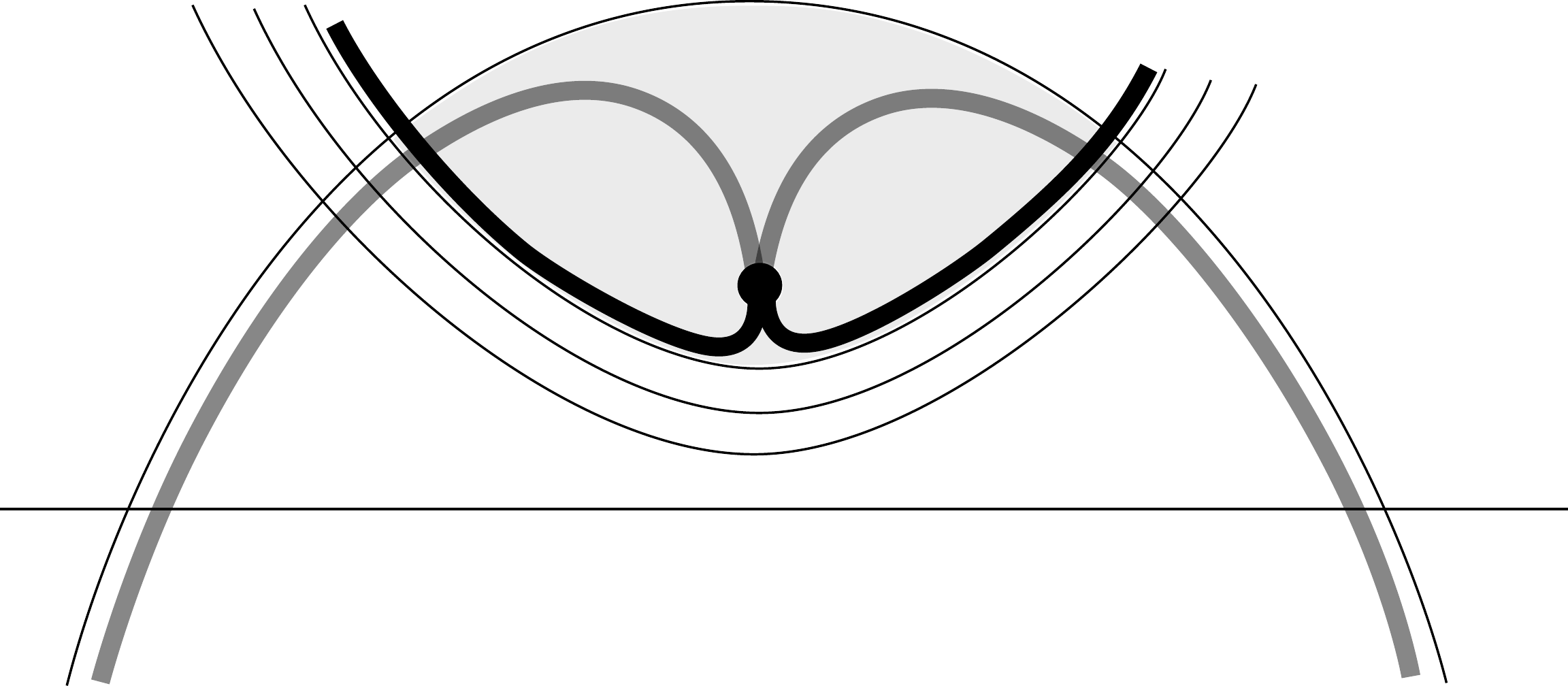}
	\subcaption{}
	\label{bigon1}
	\end{minipage} \hfill
	\begin{minipage}{.32\textwidth}
	\centering
	\includegraphics[width=5cm]{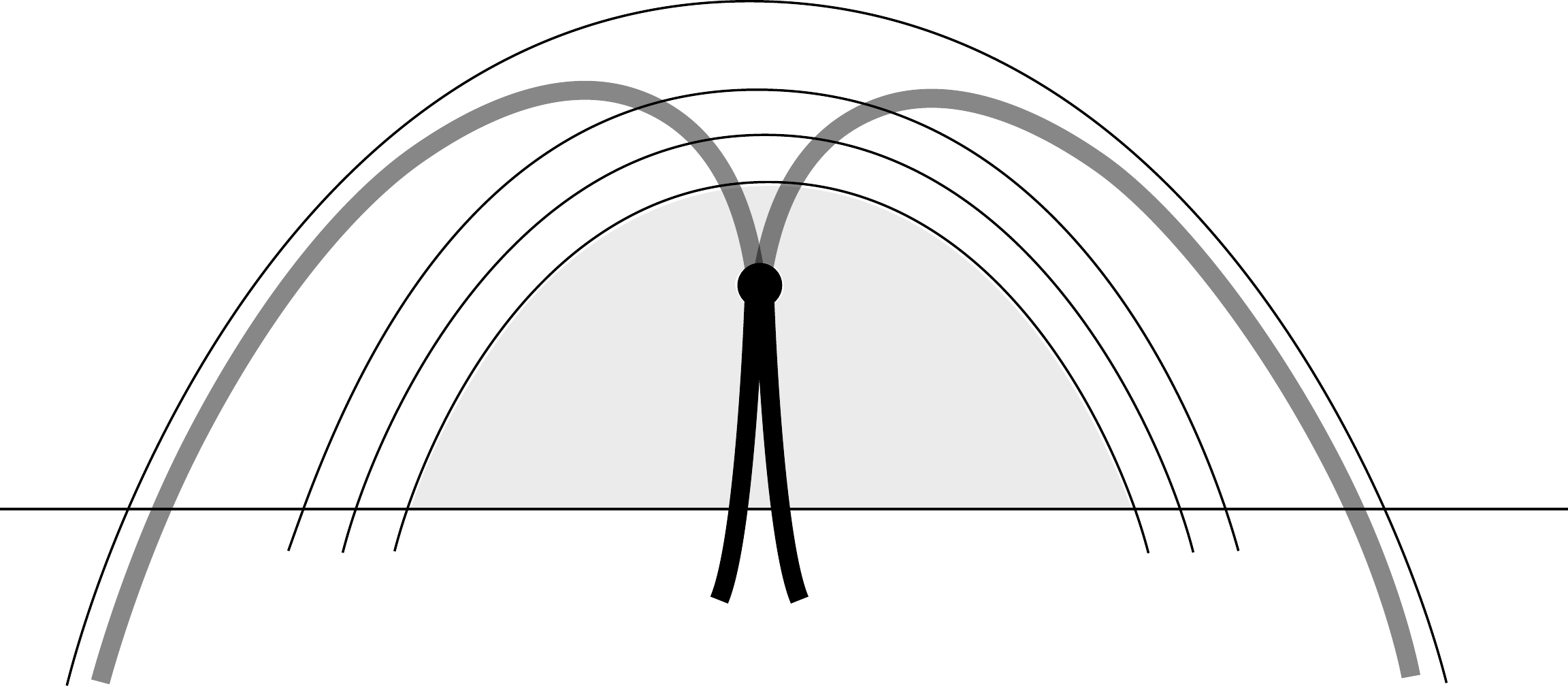}
	\subcaption{}
	\label{bigon2}
	\end{minipage}\hfill
	\begin{minipage}{.32\textwidth}
	\centering
	\includegraphics[width=5cm]{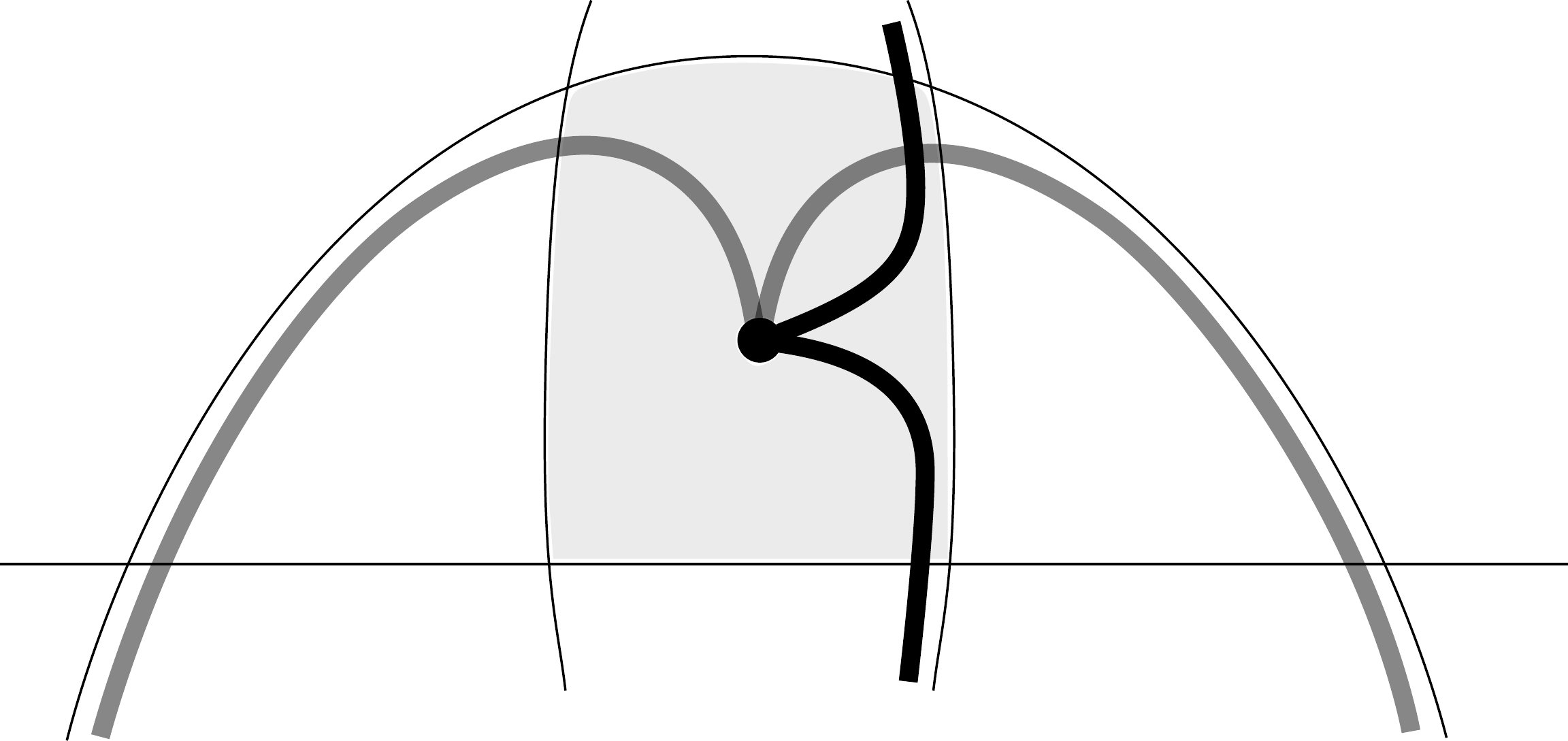}
	\subcaption{}
	\label{bigon3}
	\end{minipage}\hfill
	\caption{Possibilities for $A_c \cap A_d$ when $c$ is of type (2), with $\alpha_c$ shaded and $\alpha_d$ in bold.}
\end{figure}

\end{proof}

When $k= 2$, note the useful fact that whenever two curves $\gamma,\eta \in \Gamma$ have homotopic projections in $S_{g,n-1}$, we must have that $\gamma,\eta$  are disjoint:  
If not, there are subarcs $\gamma_p,\eta_p$ of $\gamma,\eta$ that bound a once-punctured bigon in $S_{n}$, punctured by $p$. 
The assumption that $\gamma$ and $\eta$ are homotopic after filling $p$ implies that  $\gamma \setminus \gamma_p$ and $\eta \setminus \eta_p$ must jointly bound another bigon on $S_{n-1}$ in the complement of the point $p$, and this contradicts the assumption that $\gamma$ and $\eta$ are in minimal position.
Thus, when $k= 2$, all bad curves $c\in\mathcal{B}$ are of type (1), and the proof of Lemma \ref{bad arc system} demonstrates the stronger claim:

\begin{lemma}
When $k=2$ the set $\{\alpha_c : c \in \mathcal{B} \}$ is a $1$-system of arcs.
\end{lemma}

The bound
$\displaystyle \mathcal{N}_2(S_{g,n}) \leq \mathcal{N}_2(S_{g,0}) + C(g+n)^{3}$
now follows using the same arguments as above.

\begin{remark}
\label{m-r-t method}
This method is close to that of \cite[Thm.~1.2]{m-r-t}, where they adopt the perspective that for $k=1$ the size of $|\mathcal{B}|$ can be bounded by observing that $\{A_c : c\in \mathcal{B}\}$ is a collection of annuli which pairwise intersect essentially.
\end{remark}

\subsection*{Acknowledgements}
The authors thank Piotr Przytycki, Justin Malestein, Igor Rivin, and Louis Theran for motivation and inspiration.

\end{document}